\documentclass[article,oneside,oldfontcommands,a4paper,12pt]{memoir}

\counterwithout{section}{chapter}
\setsecnumdepth{subsubsection}

\setsecheadstyle{\large\bf\raggedright}
\setsubsecheadstyle{\large\sl\raggedright}
\setsubsubsecheadstyle{\sl\raggedright}
\setsecnumformat{\csname the#1\endcsname.\;}

\usepackage[affil-it]{authblk}

\usepackage{indentfirst}
\usepackage[margin=1.05 in]{geometry}

\usepackage[numbers,sort&compress]{natbib}

\usepackage{amssymb,amsmath,mathrsfs}
\usepackage{amsthm,stmaryrd}
\usepackage[hypertexnames=false,colorlinks,linkcolor={blue},citecolor={blue}]{hyperref}
\usepackage{graphicx,float,caption}
\usepackage{enumitem}

\usepackage[T1]{fontenc}

 \usepackage[usenames,dvipsnames]{pstricks}
 \usepackage{epsfig}
 \usepackage{pst-grad} 
 \usepackage{pst-plot} 
 \usepackage[space]{grffile} 
 \usepackage{etoolbox} 
 \makeatletter 
 \patchcmd\Gread@eps{\@inputcheck#1 }{\@inputcheck"#1"\relax}{}{}
 \makeatother

\usepackage{fancyhdr}
\numberwithin{equation}{section}

\newtheorem{thm}{Theorem}[section]
\newtheorem{cor}[thm]{Corollary}
\newtheorem{lem}[thm]{Lemma}
\newtheorem{prop}[thm]{Proposition}
\theoremstyle{definition}
\newtheorem{dfn}[thm]{Definition}

\newtheorem{ques}{Question}
\theoremstyle{remark}

\allowdisplaybreaks

\newcommand{\reDeclareMathOperator}[2]{\let#1\undefined \DeclareMathOperator{#1}{#2}}

\reDeclareMathOperator{\mod}{mod}
\DeclareMathOperator{\conv}{conv}

\DeclareMathOperator{\diam}{diam}
\reDeclareMathOperator{\supp}{supp}
\DeclareMathOperator*{\argmin}{arg\,min}

\reDeclareMathOperator{\Proj}{Proj}

\DeclareMathOperator{\Fix}{Fix}

\DeclareMathOperator{\cl}{cl}

\renewcommand{\bar}{\overline}
\newcommand{\CAT}{\rm CAT}

\newcommand{\R}{\mathbb{R}}
\newcommand{\N}{\mathbb{N}}

\renewcommand{\to}{\rightarrow}

\newcommand{\tendsto}{\longrightarrow}

\newcommand{\abs}[1]{\left|#1\right|}

\renewcommand{\angle}{\measuredangle}

\makeatletter
\renewcommand*{\@fnsymbol}[1]{\ifcase#1\or*\else\@arabic{#1}\fi}
\makeatother

\title{Existence and Approximations for Order-Preserving Nonexpansive Semigroups over $\CAT(\kappa)$ Spaces}

\author{Parin Chaipunya}

\affil{
Theoretical and Computational Science Center (TaCS),\break
King Mongkut's University of Technology Thonburi, \break
126 Pracha Uthit Rd., Bang Mod, Thung Khru,\break
Bangkok 10140, Thailand.\break
Emails: parin.cha@mail.kmutt.ac.th (P. Chaipunya).
\vspace{.25cm}
}

\date{}
\begin{document}
\maketitle \vspace{-1.8cm}
\thispagestyle{empty}
\begin{abstract}

In this paper, we discuss the fixed point property for an infinite family of order-preserving mappings which satisfy the Lipschitzian condition on comparable pairs. The underlying framework of our main results is a metric space of any global upper curvature bound $\kappa \in \R$, i.e., a $\CAT(\kappa)$ space. In particular, we prove the existence of a fixed point for a nonexpasive semigroup on comparable pairs. Then, we propose and analyze two algorithms to approximate such a fixed point.

\medskip


\end{abstract}


\section{Introduction}

Metric fixed point theory was assumably started in 1922 by the work of Banach where he introduced the famous Contraction Principle with an application to Cauchy differential equations. This well-known principle applies to every complete metric spaces, and has been fruitfully extended to several generalizations of a metric space as well. To appreciate the principle, let us recall that not only the existence and uniqueness of a fixed point is guaranteed but a simple construction of such fixed point is also provided with \emph{a priori} error estimates in terms of the contractivity constant and the initial data.

As almost a century past, the subjects and objects in metric fixed point theory grows vastly, but the main theme undeniably roams around the notion of Lipschitz continuity. Suppose now that $(M,d)$ is a metric space, we say that $T : M \to M$ is \emph{Lipschitzian} if there is a constant $L \geq 0$ such that \begin{equation}\label{eqn:Lip}
d(Tx,Ty) \leq Ld(x,y)
\end{equation}
holds for any $x,y \in M$. If \eqref{eqn:Lip} holds with $L < 1$, we say that $T$ is a \emph{contraction}, while we say that $T$ is \emph{nonexpansive} if \eqref{eqn:Lip} holds with $L = 1$.

Although many generalizations of a contraction have been carried out, the naive extension to nonexpansive mappings seems not as straightforward as it looks. As a quick glance, we may take $X = \R$ and $Tx = 1+x$ for each $x \in X$. Then $T$ is nonexpansive with no fixed point. Existence theorems for nonexpansive mappings officially began in 1965, in Hilbert and certain Banach spaces, after the works of Browder, G\"ohde, and Kirk ( see \cite{MR0178324}, as well as \cite{MR0189009,MR0190718,MR0187120}). The results were generalized to a commuting family of nonexpansive mappings by DeMarr \cite{MR0159314,MR0159229}, and later improved by Lim \cite{MR0365250}.

Assuming the Lipschitzian condition \eqref{eqn:Lip} only on \emph{related} elements has set a new research stream. Ran and Reurings \cite{MR2053350} were the first to investigate such situations in the case $L < 1$ and the elements are related with a \emph{partial ordering}, \emph{i.e.,} a relation which is reflexive, antisymmetric, and transitive. Recall that if $\preceq$ is a partial ordering on a set $X$, then $x,y \in X$ are said to be $\preceq$-comparable if either $x \preceq y$ or $y \preceq x$. The results of Ran and Reurings \cite{MR2053350} were later refined and improved by Nieto and Rodr\'iguez-L\'opez \cite{MR2212687}. These fixed point results were motivated from applications to solve matrix equations and differential equations. 

The studies of nonexpansive mappings endowed with a partial ordering in Banach spaces was first considered by Bachar and Khamsi \cite{MR3368216} and was complemented with the Mann's approximation scheme in \cite{MR3402821}. The topic then extended to an order-preserving nonexpansive semigroup \cite{MR3393811} under the setting of both Banach and hyperbolic metric spaces. Here, the relationship between approximate fixed point sequences of mappings in the semigroup was thoroughly explained. After that, the full existence result for such semigroups was given in \cite{MR3872318,MR3755622} under the framework of Banach spaces, and recently in \cite{MR3830164} for the framework of hyperbolic metric spaces. Finally, an approximation result for this semigroup in Banach spaces was announced by Kozlowski \cite{KOZ} by using the Krasnosel'ski\u\i
{} process. It is important to note that Esp\'inola and Wisnicki \cite{MR3782584} recently gave a general statement that unifies all the existence results mentioned earlier in Banach spaces. The unification in hyperbolic metric spaces is not known due to an open problem about weak topologies in such spaces (see also \cite{MR3003694,MR3241330}).

Let us state the main notion for our study now. Suppose that $(X,d)$ is a metric space endowed with a partial ordering $\preceq$, and $C \subset X$ is nonempty. The family $\Gamma := \{T_{t}\}_{t \in J}$ of mappings from $C$ into itself, where $J$ is a nontrivial subsemigroup of $[0,\infty)$, is called a \emph{$\preceq$-Lipschitzian semigroup on $C$} if the following conditions are satisfied:
\begin{enumerate}[label=(S\arabic*)]
\item $T_{0} = Id_{C}$.
\item $T_{s+t} = T_{s}\circ T_{t}$ for any $s,t \geq 0$.
\item For any $x \in C$, the mapping $t \mapsto T_{t}x$ is continuous.
\item For any $t \in J$, $T_{t}$ preserves $\preceq$ in the sense that $x \preceq y$ implies $T_{t}x \preceq T_{t}y$.
\item For each $t > 0$, the inequality $d(T_{t}x,T_{t}y) \leq L d(x,y)$ holds whenever $x,y \in C$ are $\preceq$-comparable.
\end{enumerate}

In this paper, we consider a $\preceq$-nonexpansive semigroup for which $X$ is a metric space with curvature bounded above by any $\kappa \in \R$, also known as a $\CAT(\kappa)$ space. Recall that each $\CAT(\kappa)$ space is a hyperbolic metric space if $\kappa \leq 0$. It is, however, unknown to the case $\kappa > 0$.

Our main results can be broke down into three parts. First, we establish an existence result under the assumptions of $C$ being bounded, closed, and convex. Second, we propose a Krasnosel'ski\u\i{} approximation scheme, similarly to \cite{KOZ}, and show its convergence property. Note that the assumptions made in this part are only applicable to discrete (i.e., countable) semigroups. This motivates us to study the final part, where we propose the Browder approximation scheme and show appropriate convergence property. As opposed to the Krasnosel'ski\u\i{} scheme, the Browder's is implicit. However, the assumptions for the convergence are less restrictive and applies to any semigroups. The techniques used in this last part are adapted and simplified from \cite{MR2514758,MRHUANG}. Moreover, to the best of our knowledge, the Browder scheme has not yet been investigated for ordered version of Lipschitzian semigroups even in Banach or Hilbert spaces.

The organization of this paper is as follows: The next section collects all the prerequisites of $\CAT$ spaces. Sections 3--5 contains our main materials from Existence Theorems to Explicit and Implicit Approximation Schemes, respectively. The final section then concludes all the results and provides additional remarks and open questions.

\section{Preliminaries}

In this section, we shall recall the prerequisited knowledge for our main results in the next sections. We begin with the notion of geodesic metric spaces and the defining properties of $\CAT$ spaces.

Suppose that $(X,d)$ is a metric space. A \emph{geodesic} in $X$ is a curve $c : I \to X$, where $I \subset \R$ is a compact interval, and $d(c(s),c(t)) = \abs{s-t}$ holds for any $s,t \in I$. In other words, a geodesic is curve in $X$ which is isometry to some compact real interval. Without loss of generality, always assume that $I = [0,T]$ for some $T$. If $c(0) = x$ and $c(T) = y$, we say that $c$ joins $x$ and $y$. Let $D \in (-\infty,+\infty]$. If for any $x,y \in X$ with $d(x,y) < D$ are joined by a geodesic, then $X$ is said to be \emph{$D$-geodesic}. If such geodesic is unique, we further say that $X$ is \emph{$D$-uniquely geodesic}. In the latter case, we write $\llbracket x,y \rrbracket := c(I)$ to denote the (unique) geodesic segment. If $X$ is $\infty$-geodesic (or $\infty$-uniquely geodesic), we say that $X$ is geodesic (or uniquely geodesic).

If $X$ is $D$-uniquely geodesic and $c : [0,T] \to X$ joins $x$ and $y$, we write $(1-\lambda)x \oplus \lambda y := c(\lambda T)$ for $\lambda \in [0,1]$. If $C \subset X$ and $\llbracket x,y \rrbracket \subset C$ for every $x,y \in C$, then $C$ is called \emph{convex}. A function $f : C \to \R$ is called \emph{convex} if $C$ is convex and $f((1-\lambda)x\oplus \lambda y) \leq (1-\lambda)f(x) + \lambda f(y)$ for any $x,y \in C$ and $\lambda \in [0,1]$.

Let $\mathcal{M}_{\kappa}$ be the simply-connected Riemannian 2-surface of constant sectional curvature $\kappa$. Denoted by $d_{\kappa}$ the intrinsic distance function on $\mathcal{M}_{\kappa}$, $\angle^{(\kappa)}_{q}$ the angle at vertex $q \in \mathcal{M}_{\kappa}$, and by $D_{\kappa}$ the diameter of $\mathcal{M}_{\kappa}$. To be precise, we have $D_{\kappa} = \infty$ for $\kappa \leq 0$ and $D_{\kappa} = \pi/\sqrt{\kappa}$ for $\kappa > 0$. Note that $\mathcal{M}_{0} = \R^{2}$ and that $\mathcal{M}_{1} = \mathbb{S}^{2}$. To see a more detailed explanation of the subject, refer to \cite{MR1744486,MR1835418}.

Among other things, the following identity, known as the spherical law of cosines, serves as the main tool for our analsis.
\begin{prop}[\cite{MR1744486}]\label{prop:COSINE}
Suppose that $\Delta$ is a geodesic triangle in $\mathcal{M}_{\kappa}$ with $\kappa > 0$. If $\Delta$ has side lenghts $a,b,c > 0$, and $\gamma > 0$ is the angle opposite to the side with length $c$. Then,
\[
\cos (\sqrt{\kappa} c) = \cos (\sqrt{\kappa} a) \cos (\sqrt{\kappa} b) + \sin (\sqrt{\kappa} a) \sin (\sqrt{\kappa} b) \cos \gamma.
\]
\end{prop}

Fix $\kappa \in \R$ and let $X$ be $D_{\kappa}$-uniquely geodesic. For each points $p,q,r \in X$, the geodesic triangle $\Delta \subset X$ is defined by $\Delta(p,q,r) := \llbracket p,q \rrbracket \cup \llbracket q,r \rrbracket \cup \llbracket r,p \rrbracket$. The geodesic triangle $\overline{\Delta} := \Delta(\bar{p},\bar{q},\bar{r})$ with $\bar{p},\bar{q},\bar{r} \in \mathcal{M}_{\kappa}$ is said to be a \emph{$\kappa$-comparison triangle} (or simply \emph{comparison triangle}) if $d_{\kappa}(\bar{p},\bar{q}) = d(p,q)$, $d_{\kappa}(\bar{q},\bar{r}) = d(q,r)$, and $d_{\kappa}(\bar{r},\bar{p}) = d(r,p)$. Note that the triangle inequality of $d$ implies the existence of such comparison triangle. Moreover, the comparison triangle of each geodesic triangle in $X$ is unique up to rigid motions. Suppose that $\Delta(p,q,r) \subset X$ is a geodesic triangle whose comparison triangle is $\Delta(\bar{p},\bar{q},\bar{r})$. Given $u \in \llbracket p,q \rrbracket$, the point $\bar{u} \in \llbracket \bar{p},\bar{q} \rrbracket$ is said to be a \emph{comparison point} of $u$ if $d_{\kappa}(\bar{p},\bar{u}) = d(p,u)$. Comparison points for $u' \in \llbracket q,r \rrbracket$ and $u'' \in \llbracket r,p \rrbracket$ are defined likewise.

\begin{dfn}
Given $\kappa \in \R$. A $D_{\kappa}$-geodesic metric space $(X,d)$ is said to be a \emph{$\CAT(\kappa)$ space} if for each geodesic triangle $\Delta \subset X$ and two points $u,v \in \Delta$, the following $\CAT(\kappa)$ inequality holds:
\[
d(u,v) \leq d_{\kappa}(\bar{u},\bar{v}),
\]
where $\bar{u},\bar{v} \in \overline{\Delta}$ are the comparison points of $u$ and $v$, respectivly, and $\bar{\Delta} \subset \mathcal{M}_{\kappa}$ is a $\kappa$-comparison triangle of $\Delta$.
\end{dfn}

Let us give now the following fundamental facts.
\begin{lem}[\cite{MR1744486},\cite{MR2322550}]\label{lem:basicCAT(k)}
Suppose that $(X,d)$ is a complete $\CAT(\kappa)$ space. Then, the following are satisfied:
\begin{enumerate}[label=(\roman*)]
\item (\cite{MR1744486}) $X$ is also a $\CAT(\kappa')$ space for all $\kappa' \geq \kappa$.
\item (\cite{MR1744486}) For each $p \in X$, the function $d(\cdot,p)|_{B(p,D_{\kappa}/2)}$ is convex.
\item (\cite{MR2322550}) If $\kappa > 0$ and $C \subset X$ is nonempty, closed, convex, and bounded with $\diam(C) < D_{\kappa}/2$, then
\[
d^{2}(p,(1-t)x \oplus ty) \leq (1-t)d^{2}(p,x) + td^{2}(p,y) - \frac{k}{2}t(1-t)d^{2}(x,y)
\]
for $p,x,y \in C$ and $t \in [0,1]$, where $k := 2\diam(C)\tan(D_{\kappa}/2 - \diam(C))$.
\end{enumerate}
\end{lem}

Next, let $c,c'$ be two geodesics in a $\CAT(\kappa)$ space $X$ in which $c(0) = c'(0) = p$ and images of both $c$ and $c'$ are not singleton. We define the \emph{Alexandrov angle} between $c$ and $c'$ by
\[
\angle_{p}(c,c') := \limsup_{t,t' \tendsto 0^{+}} \angle^{(\kappa)}_{\bar{p}} (\bar{c(t)},\bar{c'(t')}),
\]
where $\Delta(\bar{p},\bar{c(t)},\bar{c'(t')})$ is the $\kappa$-comparison triangle of $\Delta(p,c(t),c'(t'))$ for each $t,t' > 0$ near $0$. Note that $\angle$ is symmetric and satisfies the triangle inequality whenever all the angles are defined (see \cite{MR1744486}).

In 2009, Esp\'inola and Fern\'andez-Le\'on \cite{MR2508878} studies several results related to fixed point theory and convex analysis. There are two basic results established in this paper. One is the generalization of $\Delta$-convergence to general $\CAT(\kappa)$ spaces (the concept was originally given on $\CAT(0)$ spaces earlier in \cite{MR2416076}), and the other is the well-definition of a metric projection. Now, let us recall the notions and properties of the $\Delta$-convergence.

Suppose that $(X,d)$ is a $\CAT(\kappa)$ space and $(x^{k})$ a bounded sequence in $X$. Put $\tau(x;(x^{k})) := \limsup_{k \tendsto \infty} d(x^{k},x)$ for each $x \in X$, and $A(x^{k}) := \argmin_{X} \tau(\cdot;(x^{k})$.

\begin{dfn}[\cite{MR2508878}]
A bounded sequence $(x^{k})$ in $X$ is said to be \emph{$\Delta$-convergent} to a point $\bar{x} \in X$ if $A(u^{k}) = \{\bar{x}\}$ for every subsequence $(u^{k})$ of $(x^{k})$.
\end{dfn}

\begin{prop}[\cite{MR2508878}]\label{prop:Deltalimit-Conv}
If $\inf_{x \in X} \tau(x) < D_{\kappa}/2$, then the following are satisfied:
\begin{enumerate}[label=(\roman*)]
\item $A(x^{k})$ is singleton.
\item $(x^{k})$ contains a $\Delta$-convergent subsequence.
\item If $(x^{k})$ is $\Delta$-convergent to $x \in X$, then $x \in \bigcap_{n \in \N} \cl\conv\{x^{n},x^{n+1},\dots\}$.
\end{enumerate}
\end{prop}

Finally, let us give the basic results of a metric projection in the following.

\begin{prop}[\cite{MR2508878}]\label{prop:proj}
Let $C \subset X$ be nonempty, closed, and convex, and $x \in X$ with $\inf_{c \in C} d(x,c) < D_{\kappa}/2$. Then, the following are satisfied:
\begin{enumerate}[label=(\roman*)]
\item The infimum $\inf_{c \in C} d(x,c)$ is uniquely attained. The minimizer is denoted by $\Proj_{C}(x)$.
\item If $x \not\in C$ and $y \in C\setminus\{\Proj_{C}(x)\}$, then $\angle_{\Proj_{C}(x)}(x,y) \geq \pi/2$.
\end{enumerate}
\end{prop}

\section{Existence Theorems}

In this section, we establish our first main result -- the existence of common fixed points. Before we enter the main part, let us give some notes on the partial ordering first.

\subsection{Some Introductory Notes}

Before we establish the existence of a fixed point of a semigroup $\Gamma$, we need to make an additional assumption on the partial ordering $\preceq$. In particular, we want this partial ordering $\preceq$ to be \emph{compatible} with the $\CAT$ structure of $X$ in the following sense:
\begin{enumerate}[label=(A\arabic*)]
\item For each $u \in X$, the $\preceq$-interval
\[
[u,\rightarrow) := \{z \in X \;|\; u \preceq z \}
\]
is closed.
\item If $a,b,c,d \in X$ satisfy $a \preceq b$ and $c \preceq d$, then $(1-\lambda)a \oplus \lambda c \preceq (1-\lambda)b \oplus \lambda d$ for any $\lambda \in [0,1]$.
\end{enumerate}
Note that the second assumption implies that $\preceq$-intervals are convex. Moreover, we need the following note in our further investigations.
\begin{lem}
If $a \preceq b$ and $0 \leq \lambda \leq \eta \leq 1$, then
\[
(1-\lambda)a \oplus \lambda b \preceq (1-\eta)a \oplus \eta b.
\]
\end{lem}
\begin{proof}
Notice that $(1-\lambda)a \oplus \lambda b \preceq b$. The conclusion simply follows from the fact that $(1-\eta) a \oplus \eta b \in \llbracket (1-\lambda)a \oplus \lambda b,b \rrbracket$.
\end{proof}

If $X$ is a normed linear space, the compatibility with the $\CAT$ structure is the same with compatility with the norm-topology and the linear structure. In particular, suppose that $E$ is a normed linear space. Recall that $K \subset E$ is called a \emph{cone} in $E$ if $\alpha x \in K$ for all $\alpha \geq 0$ whenever $x \in E$. Moreover, a cone $K$ is called pointed if $K \cap (-K) = \{0\}$. When $K$ is a closed convex pointed cone in $E$, we subsequently have a partial ordering $\sqsubseteq_{K}$ which is given by
\[
a \sqsubseteq_{K} b \iff b-a \in K,
\]
for $a,b \in E$. One can simply notice that $\sqsubseteq_{K}$ is compatible with the norm-topology and the linear structure in the sense given above.

It seems that the compatibility of $\preceq$ on a general $\CAT(\kappa)$ space is less obvious to be achieved. However, the $\CAT(\kappa)$ spaces which appears practical is often geometrically embedded or isometrically contained in some appropriate linear spaces, which makes the situation less complicated.

As mentioned in the introduction, Esp\'{i}nola and Wi\'{s}nicki \cite{MR3782584} recently gave a general mechanism for an order-preserving mapping in a Hausdorff topological space to have a fixed point. Their results unify several existence theorems for order-preserving mappings in the literatures assuming similar compatibility including \cite{MR3368216,MR3462097,MR3778148}. The results involving a Lipschitzian semigroup from \cite{MR3872318,MR3755622} are also similarly unified. The key ingredient in such unification is the compactness (in some topology) of the order intervals. In a reflexive normed linear space, every closed convex subset is compact in the weak topology. However, the question of whether or not there is a topology which generates $\Delta$-convergence in $\CAT$ spaces are still open (see also \cite{MR3003694,MR3241330}). It is therefore safe now to consider similar existence result in the setting of $\CAT$ spaces (or more generally the hyperbolic metric space), as was initiated in \cite{MR3830164}.

\subsection{An Existence theorem}

Throughout the rest of this paper, always assume that $(X,d)$ is a complete $\CAT(\kappa)$ space ($\kappa \in \R$) endowed with a partial ordering $\preceq$ which is compatible with the $\CAT$ structure. Assume that $C \subset X$ is nonempty, closed, convex, and bounded with $\diam(C) < D_{\kappa}/2$. Finally, assume that $\Gamma := \{T_{t}\}_{t \in J}$ is a $\preceq$-nonexpansive semigroup on $C$.

In view of Lemma \ref{lem:basicCAT(k)} and the boundedness of $C$, we can always assume that $\kappa > 0$ so that the $\kappa$-spherical law of cosines (Proposition \ref{prop:COSINE}) is applicable. The following theorem is the main existence result of this section.

\begin{thm}
The following statements are true:
\begin{enumerate}[label=(\roman*)]
\item\label{conclusion1} If there is a point $x^{0} \in X$ such that $x^{0} \preceq T_{t}x^{0}$ for all $t \in J$, then there is $w \in \Fix(\Gamma)$ such that $x^{0} \preceq w$.
\item\label{conclusion2} If $z_{1},z_{2} \in \Fix(\Gamma)$ are $\preceq$-comparable, then $\llbracket z_{1},z_{2} \rrbracket \subset \Fix(\Gamma)$.
\end{enumerate}
\end{thm}
\begin{proof}
\ref{conclusion1} The `only if' part is trivial to see. Let us proof the `if' part. Suppose that $x^{0} \in C$ satisfies $x^{0} \preceq T_{t}x^{0}$ for all $t \in J$.

Set $C_{0} := C \cap \left( \bigcap_{t \in J} [T_{t}x^{0},\rightarrow) \right)$. We claim that $C_{0}$ is nonempty, closed, and convex. The closedness and convexity of $C_{0}$ are obvious since all the sets in the intersection are closed and convex. So we only need to show that $C_{0}$ is nonempty. Indeed, suppose that $(t_{k})$ is a strictly increasing sequence in $J$ with $t_{k} \tendsto \infty$ as $k \tendsto \infty$. It follows that $(T_{t_{k}}x^{0})$ is a sequence in $C$ and is $\preceq$-nondecreasing. By the boundedness of $C$, there is a subsequence $(s_{k})$ of $(t_{k})$ in which $(T_{s_{k}}x^{0})$ is $\Delta$-convergent to some point $y \in C$. In view of Proposition \ref{prop:Deltalimit-Conv}, we have $y \in \bigcap_{n \in \N} \cl \conv \{T_{s_{n}}x^{0},T_{s_{n+1}}x^{0},\dots\}$. Since all the $\preceq$-intervals are closed and convex and $(T_{s_{n}}x^{0})$ is $\preceq$-nondecreasing, we further have 
\[
y \in C \cap \left( \bigcap_{n \in \N} \cl \conv \{T_{s_{n}}x^{0},T_{s_{n+1}}x^{0},\dots\} \right) \subset C \cap \left( \bigcap_{n \in \N} [T_{s_{n}}x^{0},\rightarrow) \right).
\]
Again, since $(T_{t}x^{0})_{t \in J}$ is $\preceq$-nondecreasing, we have the equality $\bigcap_{n \in \N} [T_{s_{n}}x^{0},\rightarrow) = \bigcap_{t \in J} [T_{t}x^{0},\rightarrow)$ and therefore the set $C_{0}$ is nonempty.


Let $p : C_{0} \to \R$ be a function defined by
\[
p(z) := \limsup_{t \tendsto \infty} d^{2}(T_{t}x^{0},z), \quad (^\forall z \in C_{0}).
\]
Note that $p$ is convex and continuous and $C_{0}$ is bounded, closed, and convex. Therefore, $p$ attains a minimizer $z^{\ast} \in C_{0}$. We may see that $T_{t}x^{0} \preceq z^{\ast}$ for all $t \geq0$. Let $s,t \in J$. By means of Lemma \ref{lem:basicCAT(k)} and the $\preceq$-nonexpansivity, we have
\begin{align*}
\lefteqn{d^{2}\left(T_{s+t}x^{0},\tfrac{1}{2}z^{\ast} \oplus \tfrac{1}{2}T_{s}z^{\ast}\right)} \\ 
&\qquad\qquad\leq \frac{1}{2}d^{2}(T_{s+t}x^{0},z^{\ast}) + \frac{1}{2}d^{2}(T_{s+t}x^{0},T_{s}z^{\ast}) - \frac{k}{8}d^{2}(z^{\ast},T_{s}z^{\ast}) \\
&\qquad\qquad\leq \frac{1}{2}d^{2}(T_{s+t}x^{0},z^{\ast}) + \frac{1}{2}d^{2}(T_{r+t}x^{0},z^{\ast}) - \frac{k}{8}d^{2}(z^{\ast},T_{s}z^{\ast}),
\end{align*}
where $k := 2\diam(C)\tan(D_{\kappa}/2 - \diam(C)) \in (0,2)$. Passing $t \tendsto \infty$ and since $z^{\ast}$ minimizes $p$, we obtain
\[
p(z^{\ast}) \leq p\left(\tfrac{1}{2}T_{r}z^{\ast} \oplus \tfrac{1}{2}T_{s}z^{\ast}\right) \leq p(z^{\ast}) - \frac{k}{8}d^{2}(z^{\ast},T_{s}z^{\ast}).
\]
This implies $z^{\ast} = T_{s}z^{\ast}$ for all $s \in J$. Moreover, since $(T_{s_{n}}x^{0})$ is $\preceq$-nondecreasing and we have $x^{0} \preceq z^{\ast}$.

\ref{conclusion2} Suppose that $z_{1},z_{2} \in \Fix(\Gamma)$ and that $z_{1} \preceq z_{2}$. We may also assume that $z_{1} \neq z_{2}$, since the conclusion is immediate otherwise. Let $c \in [0,1]$ and put $z := (1-c)z_{1} \oplus cz_{2}$. By the assumption on $\preceq$, we have $z_{1} \preceq z \preceq z_{2}$. For $t \in J$, we have
\begin{equation}\label{eqn:cvx1}
d(T_{t}z,z_{1}) = d(T_{t}z,T_{t}z_{1}) \leq d(z,z_{1}) = cd(z_{1},z_{2})
\end{equation}
and also
\begin{equation}\label{eqn:cvx2}
d(T_{t}z,z_{2}) = d(T_{t}z,T_{t}z_{2}) \leq d(z,z_{2}) = (1-c)d(z_{1},z_{2}).
\end{equation}
Using the triangle inequality, we get
\[
d(z_{1},z_{2}) \leq d(z_{1},T_{t}z) + d(T_{t}z,z_{2}) \leq cd(z_{1},z_{2}) + (1-c)d(z_{1},z_{2}) = d(z_{1},z_{2}).
\]
This means $d(z_{1},T_{t}z) + d(T_{t}z,z_{2}) = d(z_{1},z_{2})$, which implies that $T_{t}z \in \llbracket z_{1},z_{2} \rrbracket$. So $T_{t}z = (1-c')z_{1} \oplus c'z_{2}$ for some $c' \in [0,1]$. Moreover, we have $d(T_{t}z,z_{1}) = c'd(z_{1},z_{2})$ and $d(T_{t}z,z_{2}) = (1-c')d(z_{1},z_{2})$ for some $c' \in [0,1]$. Together with \eqref{eqn:cvx1} and \eqref{eqn:cvx2}, we obtain $c' \leq c$ and $1-c' \leq 1-c$ which yeilds $c' = c$. It follows that $T_{t}z = z$ for every $t \in J$. Since $c \in [0,1]$ is taken arbitrarily, we conclude that $\llbracket z_{1},z_{2} \rrbracket \subset \Fix(\Gamma)$.
\end{proof}

We immediately have the following consequence. Note that this consequence is also new in the setting of a $\CAT(\kappa)$ space.
\begin{cor}
Suppose that $\Gamma$ is a nonexpansive semigroup. Then $\Fix(\Gamma)$ is nonempty, closed, and convex.
\end{cor}

\section{Explicit Approximation Scheme}

After we have proved the existence of a fixed point for the semigroup $\Gamma$ in the previous section, we hereby propose an algorithm to approximate such a solution. The algorithm presented in this section is a modification of the Krasnosel'ski\u\i{} approximation schemes.

Let us now give the formal definition of the Krasnosel'ski\u\i{} approximation scheme associated with $\Gamma$ as follows: Let $\lambda \in (0,1)$ and $(t_{k})$ be a strictly increasing positive real sequence such that $t \tendsto \infty$ as $k \tendsto \infty$. Suppose that $x^{0} \in X$ has the property $x^{0} \preceq T_{t}x^{0}$ for all $t \in J$, generate for each $k \in \N$ the successive point
\begin{equation}\label{eqn:KM}
x^{k+1} := (1-\lambda)x^{k} \oplus \lambda T_{t_{k}}x^{k}.
\end{equation}
For this section, always suppose that $(x^{k})$ is the sequence given by \eqref{eqn:KM} from a point $x^{0} \in X$. We shall also refer to this sequence as the \emph{Krasnosel'ski\u\i{} sequence generated from $x^{0}$}.

We shall decompose the proof for the convergence of $(x^{k})$ into a number of Lemmas as stated in the following.
\begin{lem}\label{lem:fact1}
The following assertions hold for each $k \in \N$.
\begin{enumerate}[label=(\roman*)]
\item $x^{k} \preceq x^{k+1}$.
\item $x^{k} \preceq T_{s}x^{k}$ for $s \in J$ with $s \geq t_{k}$.
\item $x^{k} \preceq T_{t_{k}}x^{k}$.
\end{enumerate}
\end{lem}
\begin{proof}
Following from $x^{0} \preceq T_{t_{0}}x^{0}$, we get
\[
x^{0} \preceq (1-\lambda)x^{0} \oplus \lambda T_{t_{0}}x^{0} \preceq T_{t_{0}}x^{0} \preceq T_{s}x^{0}
\]
for all $s \in J$ with $s \geq t_{0}$. This shows that $x^{0} \preceq x^{1} \preceq T_{s}x^{0}$ for all $s \in J$ with $s \geq t_{0}$. In particular, we have $x^{0} \preceq T_{t_{1}}x^{0}$. The conclusion follows by the induction process.
\end{proof}

\begin{lem}\label{lem:fact2}
If $w \in \Fix(\Gamma)$ satisfies $x^{0} \preceq w$, then the limit $\lim_{k \tendsto \infty} d(w,x^{k})$ exists
\end{lem}
\begin{proof}
Since $T_{t}$ preserves $\preceq$ and $x^{0} \preceq w$, we have $T_{t}x^{0} \preceq T_{t}w = w$ for each $t \in J$. It follows from Lemma \ref{lem:fact1} that $x^{k} \preceq T_{t_{k}}x^{0} \preceq w$ for all $k \in \N$. Next, observe that the $\preceq$-nonexpansivity yields
\begin{align*}
d(w,x^{k+1}) &= d(w,(1-\lambda)x^{k} \oplus \lambda T_{t_{k}}x^{k}) \\
&\leq (1-\lambda)d(w,x^{k}) + \lambda(w,T_{t{k}}x^{k}) \\
&\leq d(w,x^{k}).
\end{align*}
Therefore, the sequence $\left(d(w,x^{k})\right)$ is nonincreasing and bounded from below. This shows that the desired limit exists.
\end{proof}

\begin{lem}\label{lem:fact3}
The following limits hold.
\begin{enumerate}[label=(\roman*)]
\item\label{cdn:fact2-1} $\lim_{k \tendsto \infty} d(x^{k},T_{t_{k}}x^{k}) = 0$.
\item\label{cdn:fact2-2} $\lim_{k \tendsto \infty} d(x^{k},x^{k+1}) = 0$.
\end{enumerate}
\end{lem}
\begin{proof}
Let $w \in \Fix(\Gamma)$ satisfies $x^{0} \preceq w$.

\ref{cdn:fact2-1} Observe for each $k \in \N$ the following
\begin{align*}
d^{2}(w,x^{k+1}) &= d^{2}(w,(1-\lambda)x^{k} \oplus \lambda T_{t_{k}}x^{k}) \\
&\leq (1-\lambda)d^{2}(w,x^{k}) + \lambda d^{2}(w,T_{t_{k}}x^{k}) - \frac{k_{2}}{2}\lambda(1-\lambda)d^{2}(x^{k},T_{t_{k}}x^{k}) \\
&\leq d^{2}(w,x^{k}) - \frac{k_{2}}{2}\lambda(1-\lambda)d^{2}(x^{k},T_{t_{k}}x^{k}).
\end{align*}
Letting $k \tendsto \infty$ and put $r := \lim_{k \tendsto \infty} d(w,x^{k})$, we get
\[
r^{2} \leq r^{2} - \frac{k_{2}}{2}\lambda(1-\lambda) \limsup_{k \tendsto \infty} d^{2}(x^{k},T_{t_{k}}x^{k}).
\]
It follows that $\lim_{k \tendsto \infty} d(x^{k},T_{t_{k}}x^{k}) = 0$.

\ref{cdn:fact2-2} Since $d(x^{k},x^{k+1}) = d(x^{k},(1-\lambda)x^{k} \oplus \lambda T_{t_{k}}x^{k}) = \lambda d(x^{k},T_{t_{k}}x^{k})$, the conclusion follows from \ref{cdn:fact2-1}.
\end{proof}

From this point, we need to assume additional conditions on the construction of the sequence $(t_{k})$ in relation with the overall structure of the semigroup $J$. This condition is strong but it allows us to obtain the approximate fixed point sequence.

\begin{lem}\label{lem:fact4}
Assume that $s \in J$ has the following property: 
\begin{equation}\label{eqn:cdn:sjk}
\begin{array}{c}
\text{there exists a strictly increasing sequence $(j_{k})$ of positive integers such that}\\ \text{$t_{j_{k}} = s + t_{j_{k}}, \quad (^\forall k \in \N)$.}
\end{array}
\end{equation}
Then $(x^{j_{k}})$ is an approximate fixed point sequence of $T_{s}$, i.e., $\lim_{k \tendsto \infty} d(x^{j_{k}},T_{s}x^{j_{k}}) = 0$.
\end{lem}
\begin{proof}
Suppose that $k \in \N$ is sufficiently large so that $t_{j_{k}} > s$. In view of \ref{lem:fact1} and the $\preceq$-nonexpansivity of $T_{s}$, we have
\begin{align*}
d(x^{j_{k+1}},T_{s}x^{j_{k+1}}) &\leq d(x^{j_{k+1}},T_{t_{j_{k+1}}}x^{j_{k+1}}) + d(T_{t_{j_{k+1}}}x^{j_{k+1}},T_{s}x^{j_{k+1}}) \\
&= d(x^{j_{k+1}},T_{t_{j_{k+1}}}x^{j_{k+1}}) + d(T_{s}T_{t_{j_{k}}}x^{j_{k+1}},T_{s}x^{j_{k+1}}) \\
&\leq d(x^{j_{k+1}},T_{t_{j_{k+1}}}x^{j_{k+1}}) + d(T_{t_{j_{k}}}x^{j_{k+1}},x^{j_{k+1}}).
\end{align*}
Letting $k \tendsto \infty$ and apply Lemma \ref{lem:fact3}, we get $\lim_{k \tendsto \infty} d(x^{j_{k}},T_{s}x^{j_{k}}) = 0$.
\end{proof}

\begin{lem}\label{lem:fact5}
Suppose that $s \in J$ has the property \eqref{eqn:cdn:sjk} and assume further that $\sup_{k \in \N} (j_{k} - k) < \infty$. Then, the following assertions hold.
\begin{enumerate}[label=(\roman*)]
\item\label{cdn:fact5-1} $\lim_{k \tendsto \infty} d(x^{k},x^{j_{k}}) = 0$.
\item\label{cdn:fact5-2} $(x^{k})$ is an approximate fixed point sequence for $T_{s}$.
\end{enumerate}
\end{lem}
\begin{proof}
\ref{cdn:fact5-1} Put $P := \sup_{k \in \N} (j_{k} - k)$. If $P = 0$, the conclusion is already verified. Hence, assume that $P > 0$. Let $\varepsilon > 0$ be chosen arbitrarily. From Lemma \ref{lem:fact3}, we know that $d(x^{k},x^{k+1}) < \varepsilon/P$ holds for any $k$ sufficiently large. For such large $k \in \N$, we have
\[
d(x^{k},x^{j_{k}}) \leq \sum_{i=k}^{j_{k}-1} d(x^{i},x^{i+1}) < (j_{k} - k)\frac{\varepsilon}{P} \leq P \cdot \frac{\varepsilon}{P} = \varepsilon.
\]
This proves $\lim_{k \tendsto \infty} d(x^{k},x^{j_{k}}) = 0$.

\ref{cdn:fact5-2} By Lemma \ref{lem:fact1} and the $\preceq$-nonexpansivity, we have
\begin{align*}
d(x^{k},T_{s}x^{k}) &\leq d(x^{k},x^{j_{k}}) + d(x^{j_{k}},T_{s}x^{j_{k}}) + d(T_{s}x^{j_{k}},T_{s}x^{k}) \\
&\leq d(x^{k},x^{j_{k}}) + d(x^{j_{k}},T_{s}x^{j_{k}}) + d(x^{j_{k}},x^{k})
\end{align*}
Letting $k \tendsto \infty$, apply the earlier fact \ref{cdn:fact5-1} and Lemma \ref{lem:fact4}, we have
\[
\lim_{k \tendsto \infty} d(x^{k},T_{s}x^{k}) = 0,
\]
which is the desired result.
\end{proof}

After having gathered all the technical lemmas required for the convergence result, we now state and prove the main theorem of this section.

\begin{thm}
Assume that all $s \in J$ has the property \eqref{eqn:cdn:sjk} with $\sup_{j_{k}-k} < \infty$. Then, the Krasnosel'ski\u\i{} sequence $(x^{k})$ generated from $x^{0}$ is $\Delta$-convergent to a point $w \in \Fix(\Gamma)$ with $x^{0} \preceq w$.
\end{thm}
\begin{proof}
First, note that the boundedness of $C$ implies the boundedness of $(x^{k})$. So $(x^{k})$ contains a $\Delta$-convergent subsequence. Suppose that $(y^{k})$ and $(z^{k})$ be two $\Delta$-convergent subsequences of $(x^{k})$ whose $\Delta$-limits are $y^{\ast}$ and $z^{\ast}$, respectively. Suppose that $y^{\ast} \neq z^{\ast}$.

Since $(x^{k})$ is $\preceq$-nondecreasing, we have $y^{k} \preceq y^{\ast}$ and $z^{k} \preceq z^{\ast}$ for all $k \in \N$. Let $s \in J$. From Lemma \ref{lem:fact5}, we have
\begin{align*}
\limsup_{k \tendsto \infty} d(y^{k},T_{s}y^{\ast}) & \leq \limsup_{k \tendsto \infty} d(y^{k},T_{s}y^{k}) + \limsup_{k \tendsto \infty} d(T_{s}y^{k},T_{s}y^{\ast}) \\
&\leq \limsup_{k \tendsto \infty} d(y^{k},y^{\ast}).
\end{align*}
Since $y^{\ast}$ is the unique asymptotic center of $(y^{k})$, it follows that $y^{\ast} = T_{s}y^{\ast}$. With the same arguments, we also have $z^{\ast} = T_{s}z^{\ast}$. Since $s \in J$ is arbitrary, we have $y^{\ast},z^{\ast} \in \Fix(\Gamma)$ with $x^{0} \preceq y^{\ast}$ and $x^{0} \preceq z^{\ast}$. Set $r_{1} := \lim_{k \tendsto \infty} d(x^{k},y^{\ast})$ and $r_{1} := \lim_{k \tendsto \infty} d(x^{k},z^{\ast})$, where the existence of such limits follows from Lemma \ref{lem:fact2}. From the fact that $(y^{k})$ and $(z^{k})$ are subsequences of $(x^{k})$ and the uniqueness of the asymptotic center, we have
\[
r_{1} = \lim_{k \tendsto \infty} d(x^{k},y^{\ast}) = \lim_{k \tendsto \infty} d(y^{k},y^{\ast}) < \limsup_{k \tendsto \infty} d(y^{k},z^{\ast}) = r_{2}
\]
and also
\[
r_{2} = \lim_{k \tendsto \infty} d(x^{k},z^{\ast}) = \lim_{k \tendsto \infty} d(z^{k},z^{\ast}) < \lim_{k \tendsto \infty} d(z^{k},y^{\ast}) = r_{2}.
\]
This gives a contradiction, and therefore it must be the case that $y^{\ast} = z^{\ast}$. In other words, $(x^{\ast})$ has only one $\Delta$-accumulation point, denoted with $w$. Similarly, we have $x^{k} \preceq w$ for all $k \in \N$.

Let $s \in J$. The Lemma \ref{lem:fact5} yields
\begin{align*}
\limsup_{k \tendsto \infty} d(x^{k},T_{s}w) &\leq \limsup_{k \tendsto \infty} d(x^{k},T_{s}x^{k}) + \limsup_{k \tendsto \infty} d(T_{s}x^{k},T_{s}w) \\
&\leq \limsup_{k \tendsto \infty} d(x^{k},w).
\end{align*}
The uniqueness of the asymptotic center guarantees that $w = T_{s}w$ and further that $w \in \Fix(\Gamma)$. Additionally, the fact that $(x^{k})$ is $\preceq$-nondecreasing yields $x^{0} \preceq w$.
\end{proof}

\section{Implicit Approximation Scheme}

In the previous section, we deals with the Krasnosel'ski\u\i{} approximation schemes where the computation of each iterate can be carried out explicitly by a specific formula. In this section, we present another route to approximate a solution $w \in \Fix(\Gamma)$ by using the Browder approximation schemes which is of different nature to the Krasnosel'ski\u\i{} approximation schemes. In the Browder approximation scheme, there is no specific closed form for each iterate. However, it can be simply computed by the use of Picard's procedure.

Also note again that we have not seen Browder approximation in this setting even when the space is linear. Since a Hilbert space is $\CAT(0)$, our next main theorem applies there.

The construction and several properties of the algorithm studied in this section are based on a theorem of Nieto and Rodr\'iguez-L\'opez \cite{MR2212687}.

\begin{thm}[\cite{MR2212687}]\label{thm:N-RL}
Let $(X,d)$ be a complete metric space that is endowed with a partial ordering $\preceq$ with the following property:
\begin{equation}\label{eqn:closedordering}
\text{If a $\preceq$-nondecreasing sequence $(x^{k})$ in $X$ converges to $x^{\ast}$, then $x^{k} \preceq x^{\ast}$ for each $k \in \N$.}
\end{equation}
Suppose that $T : X \to X$ is a mapping in which \eqref{eqn:Lip} holds for each $x,y \in X$ that are $\preceq$-comparable with $L < 1$. If there is a point $x^{0} \in X$ such that $x^{0} \preceq Tx^{0}$, then
\begin{enumerate}[label=(\roman*)]
\item $T$ has a fixed point.
\item The orbit $(T^{k}x^{0})$ converges to a fixed point $w \in \Fix(T)$.
\item $T^{k}x^{0} \preceq w$ for all $k \in \N$.
\end{enumerate}
\end{thm}

Recall again that $\Gamma := \{T_{t}\}_{t \in J}$ is a $\preceq$-nonexpansive semigroup on a bounded closed convex set $C \subset X$, and $x^{0} \in C$ is fixed with the property $x^{0} \preceq T_{t}x^{0}$ for all $t \in J$. At each $t \in J$ and $\lambda \in (0,1)$, we define $T_{t}^{\lambda} := (1-\lambda) T_{t} \oplus \lambda x^{0}$.

Let us now give the following simple facts which are essential in our main construction in this section.

\begin{lem}\label{lem:factB1}
For each $\lambda \in (0,1)$ and $t \in J$, the following facts hold.
\begin{enumerate}[label=(\roman*)]
\item\label{cdn:factB1-1} $T_{t}^{\lambda}$ is a $\preceq$-contraction with constant $(1-\lambda)$.
\item\label{cdn:factB1-2} $T_{t}^{\lambda}$ is $\preceq$-nondecreasing
\item\label{cdn:factB1-3} $x^{0} \preceq T_{t}^{\lambda} x^{0} \preceq T_{t}x^{0}$.
\end{enumerate}
\end{lem}
\begin{proof}
\ref{cdn:factB1-1} Let $x,y \in X$ with $x \preceq y$. We have
\begin{align*}
d(T_{t}^{\lambda}x,T_{t}^{\lambda}y) &= d((1-\lambda)T_{t}x \oplus \lambda x^{0},(1-\lambda)T_{t}y \oplus \lambda x^{0}) \\
&\leq (1-\lambda)d(T_{t}x,T_{t}y) \leq d(x,y).
\end{align*}
This shows the $\preceq$-contractivity of $T_{t}^{\lambda}$.

\ref{cdn:factB1-2} Let $x,y \in X$ with $x \preceq y$. Since $T_{t}$ is $\preceq$-nondecreasing, it is immediate to see that
\[
T_{t}^{\lambda} x = (1-\lambda)T_{t}x \oplus \lambda x^{0} \preceq (1-\lambda)T_{t}y \oplus \lambda x^{0} = T_{t}^{\lambda} y.
\]

\ref{cdn:factB1-3} Since $x^{0} \preceq T_{t}x^{0}$, we have $x^{0} \preceq (1-\lambda) T_{t}x^{0} \oplus x^{0} = T_{t}^{\lambda} x^{0} \preceq T_{t}x^{0}$.
\end{proof}

The following fact is obvious from the results aforestated. However, we collect it here for convenience and explicity.
\begin{lem}\label{lem:factB2}
Let $\lambda \in (0,1)$ and $t \in J$. Then, $\lim_{n \tendsto \infty} (T_{t}^{\lambda})^{n} = x_{t}^{\lambda} \in \Fix(T_{t}^{\lambda})$ with $(T_{t}^{\lambda}x^{0})^{n} \preceq x_{t}^{\lambda}$ for all $n \in \N$.
\end{lem}
\begin{proof}
Since all $\preceq$-intervals are closed, the condition \eqref{eqn:closedordering} is satisfied. Apply Lemma \ref{lem:factB1} and Theorem \ref{thm:N-RL} to arrive at the conclusion.
\end{proof}

Now, let us define the Browder approximation associated with $\Gamma$. Suppose that $(\lambda_{k})$ a strictly decreasing sequence in $(0,1)$ and $(t_{k})$ is a strictly increasing sequence of positive reals. In this situation, we adopt the notions $T^{[k]} := T_{t_{k}}^{\lambda_{k}}$ for each $k \in \N$. Next, generate for each $k \in \N$ the successive point
\[
x^{k} := \lim_{n \tendsto \infty} (T^{[k]})^{n} x^{0}.
\]
In this case, the sequence $(x^{k})$ is called the \emph{Browder sequence generated from $x^{0}$}. One may observe from Lemmas \ref{lem:factB1} and \ref{lem:factB2} that for each $k \in \N$, $x^{k} \in \Fix(T^{[k]})$ and $(T^{[k]})^{n}x^{0} \preceq x^{k}$ for all $n \in \N$. Also, we can see that $((T^{[k]})^{n}x^{0})$ is $\preceq$-nondecreasing.

For a technical reason, assume throughout this section that $t^{0} \in J\setminus\{0\}$ and $t^{k+1} := 2t^{k}$ for $k \in \N$.

\begin{lem}\label{lem:factB3}
The following assertions hold for each $k \in \N$.
\begin{enumerate}[label=(\roman*)]
\item\label{cdn:factB3-1} $x^{k} \preceq x^{k+1}$.
\item\label{cdn:factB3-2} $x^{k} \preceq T_{t_{k}}x^{k}$.
\end{enumerate}
\end{lem}
\begin{proof}
\ref{cdn:factB3-1} Fix $k \in \N$. Since $T_{t_{k}}$ is $\preceq$-nondecreasing, we apply Lemma \ref{lem:factB1} and obtain
\[
T_{t_{k}}T^{[k]}x^{0} \preceq T_{t_{k}}T_{t_{k}}x^{0} = T_{2t_{k}}x^{0} = T_{t_{k+1}}x^{0} \preceq T_{t_{k+1}}T^{[k+1]}x^{0}.
\]
Again, since $T_{t_{k}}$ is $\preceq$-nondecreasing and $(\lambda_{k})$ is strictly decreasing, we further have
\begin{align*}
(T^{[k]})^{2}x^{0} &= (1-\lambda_{k})T_{t_{k}}T^{[k]}x^{0} \oplus \lambda_{k} x^{0} \preceq (1-\lambda_{k})T_{t_{k+1}}T^{[k+1]}x^{0} \oplus \lambda_{k} x^{0} \\
&\preceq (1-\lambda_{k+1})T_{t_{k+1}}T^{[k+1]}x^{0} \oplus \lambda_{k+1} x^{0} = (T^{[k+1]})^{2}x^{0}.
\end{align*}
Now, let $n \in \N$ be an integer such that the statement $(T^{[k]})^{n}x^{0} \preceq (T^{[k+1]})^{n} x^{0}$ holds true. We may observe using similar facts that
\begin{align*}
T_{t_{k}}(T^{[k]})^{n}x^{0} &\preceq T_{t_{k}}T_{t_{k}}(T^{[k]})^{n-1}x^{0} = T_{t_{k+1}} (T^{[k]})^{n-1}x^{0}\\
&\preceq T_{t_{k+1}}(T^{[k]})^{n-1}T^{[k]}x^{0} = T_{t_{k+1}}(T^{[k]})^{n}x^{0}\\
&\preceq T_{t_{k+1}}(T^{[k+1]})^{n}x^{0}.
\end{align*}
Similarly, using the facts that $T_{t_{k}}$ is $\preceq$-nondecreasing and $(\lambda_{k})$ is strictly decreasing, we get
\begin{align*}
(T^{[k]})^{n+1}x^{0} &= (1-\lambda_{k})T_{t_{k}}(T^{[k]})^{n}x^{0} \oplus \lambda_{k} x^{0} \preceq (1-\lambda_{k})T_{t_{k+1}}(T^{[k+1]})^{n}x^{0} \oplus \lambda_{k} x^{0} \\
&\preceq (1-\lambda_{k+1})T_{t_{k+1}}(T^{[k+1]})^{n}x^{0} \oplus \lambda_{k+1} x^{0} = (T^{[k+1]})^{n+1}x^{0}.
\end{align*}
Hence, mathematical induction implies
\begin{equation}\label{eqn:compareSeq}
(T^{[k]})^{n}x^{0} \preceq (T^{[k+1]})^{n} x^{0}
\end{equation}
for every $n \in \N$.

Next, recall that Theorem \ref{thm:N-RL} gives $\lim_{n \tendsto \infty} (T^{[k]})^{n}x^{0} = x^{k}$ and $(T^{[k]})^{n}x^{0} \preceq x^{k}$ for all $n,k \in \N$. Taking \eqref{eqn:compareSeq} into account, we see now that $((T^{[k]})^{n}x^{0})$ is a sequence in the $\preceq$-interval $(\leftarrow,x^{k+1}]$, which is a closed set. Therefore, the point $x^{k}$ belongs to $(\leftarrow,x^{k+1}]$ as the limit of $((T^{[k]})^{n}x^{0})$ and we conclude here that $x^{k} \preceq x^{k+1}$ for any $k \in \N$.

\ref{cdn:factB3-2} Fix $k \in \N$. Since $x^{k} \in \Fix(T^{[k]})$, we have
\[
x^{k} = (1-\lambda_{k})T_{t_{k}}x^{k} \oplus \lambda_{k} x^{0}.
\]
Recall that $x^{0} \preceq x^{k}$ and $x^{0} \preceq T_{s}x^{0}$ for all $s \in J$. If $t \in J$ and $t \geq t_{k}$, we have $x^{0} \preceq T_{t_{k}}x^{0} \preceq T_{t_{k}}x^{k}$, which further yields
\[
x^{k} = (1-\lambda_{k})T_{t_{k}}x^{k} \oplus \lambda_{k} x^{0} \preceq (1-\lambda_{k})T_{t_{k}}x^{k} \oplus \lambda_{k} T_{t_{k}}x^{k} = T_{t_{k}}x^{k}. \tag*{\qedhere}
\]
\end{proof}

Before we go further, let us consider for a while an ordinary metric space $(Y,p)$ and a family $\Xi := \{S_{t}\}_{t \in J}$ of self-mappings on a bounded subset $K \subset Y$, indexed by a nontrivial subsemigroup $J$ of $[0,\infty)$. The following notions and lemma are variants to the similar definition given by Huang \cite{MRHUANG} for which $J$ is not necessarily the same as $[0,\infty)$. The proof is carried out in the same way so we leave it to the reader.

\begin{dfn}
The family $\Xi$ is called \emph{asymptotically regular} (or briefly, \emph{AR}) if for any $h \in J$ and $y \in K$, the following limit holds:
\[
\lim_{\substack{t \in J \\ t \tendsto \infty}} d(T_{t}y,T_{h}T_{t}y) = 0.
\]
Moreover, it is called \emph{uniformly asymptotically regular} (or brieftly, \emph{UAR}) if for any $h \in J$, the following limit holds:
\[
\lim_{\substack{t \in J \\ t \tendsto \infty}} \sup_{y \in K} d(T_{t}y,T_{h}T_{t}y) = 0.
\]
\end{dfn}

\begin{lem}\label{lem:AR-Fix}
If $\Xi$ is AR and $S_{t}S_{t'} = S_{t+t'}$ for $t,t' \in J$, then $\Fix(\Xi) = \Fix(S_{t})$ for any $t \in J$.
\end{lem}

Now we get back to our main result.
\begin{thm}
Assume that $(\lambda_{k})$ is a strictly decreasing sequence in $(0,1)$ with the limit $\lim_{k \tendsto \infty} \alpha_{k} = 0$, and $(t_{k})$ is a sequence given by $t_{k+1} = 2t_{k}$ for $k \in \N$ with $t_{0} \in J\setminus\{0\}$. Also suppose that $\Gamma$ is UAR. Then, the Browder sequence converges strongly to $y \in \Fix(\Gamma)$ with $x^{0} \preceq y$. Moreover, if $q \in \Fix(\Gamma)$ satisfies $v^{k} \preceq q$ at each $k \in \N$ for some subsequence $(v^{k})$ of $(x^{k})$, then $d(x^{0},y) \leq d(x^{0},q)$.
\end{thm}
\begin{proof}
Note first that if $x^{0} \in \Fix(\Gamma)$, then $x^{k} = x^{0}$ for all $k \in \N$. Now, consider the case $x^{0} \not\in \Fix(\Gamma)$. Since $C$ is bounded, $(x^{k})$ contains a subsequence $(y^{k})$ which is $\Delta$-convergent to some point $y \in C$. Note that $y^{k} \preceq y$ for any $k \in \N$. Suppose that $(\beta_{k})$ and $(s_{k})$ are respective subsequences of $(\lambda_{k})$ and $(t_{k})$ for which $y^{k} = (1-\beta_{k})T_{s_{k}}y^{k} \oplus \beta_{k} x^{0}$ for all $k \in \N$. Fix any $t \in J$. Then, Lemma \ref{lem:factB3} and the convexity of $d$ on $C$ implies
\begin{align*}
d(T_{t}y,y^{k}) &\leq d(T_{t}y,T_{t}y^{k}) + d(T_{t}y^{k},T_{t}T_{t_{k}}y^{k}) + d(T_{t}T_{t_{k}}y^{k},y^{k})\\
&\leq d(y,y^{k}) + d(y^{k},T_{t_{k}}y^{k}) + \beta_{k}d(x^{0},T_{t}T_{t_{k}}y^{k}) + (1-\beta_{k})d(T_{t_{k}}y^{k},T_{t}T_{t_{k}}y^{k}) \\
&= d(y,y^{k}) + \beta_{k}d(x^{0},T_{t_{k}}y^{k}) + \beta_{k}d(x^{0},T_{t}T_{t_{k}}y^{k}) + (1-\beta_{k})d(T_{t_{k}}y^{k},T_{t}T_{t_{k}}y^{k}).
\end{align*}
Letting $k \tendsto \infty$, from $\lim_{k \tendsto \infty} \beta_{k} = 0$ and $\Gamma$ being UAR, we get
\[
\limsup_{k \tendsto \infty} d(T_{t}y,y^{k}) \leq \limsup_{k \tendsto \infty} d(y,y^{k}).
\]
By the uniqueness of the asymptotic center, we have $y \in \Fix(T_{t})$. Lemma \ref{lem:AR-Fix} implies further that $y \in \Fix(\Gamma)$.

Next, we claim that $(y^{k})$ contains a strongly convergent subsequence. Let us suppose to the contrary that $\limsup_{k \tendsto \infty} d(y,y^{k}) = \sigma > 0$. For $k \in \N$, since $y^{k} \preceq y = T_{s_{k}}y$, we have
\begin{align*}
d(y,y^{k}) &\leq \beta_{k}d(y,x^{0}) + (1-\beta_{k})d(y,T_{s_{k}}y^{k}) \\
&\leq \beta_{k}d(y,x^{0}) + (1-\beta_{k})d(y,y^{k}).
\end{align*}
Passing $k \tendsto \infty$, one obtain
\[
\limsup_{k \tendsto \infty} d(y,T_{s_{k}}y^{k}) = \sigma.
\]
By passing to a subsequence, we assume without the loss of generality that $y^{k} \neq x^{0}$ for all $k \in \N$. Recall that
\[
d(x^{0},y^{k}) = (1-\beta_{k})d(x^{0},t_{s_{k}}y^{k}) < d(x^{0},T_{s_{k}}y^{k}).
\]
Since $y^{k} \in \llbracket x^{0},T_{s_{k}}y^{k} \rrbracket$, we have
\[
d(y^{k},T_{s_{k}}y^{k}) = d(x^{0},T_{s_{k}}y^{k}) - d(x^{0},y^{k}) > 0.
\]
Note that $x^{0} \neq y$ since $x^{0} \not\in \Fix(\Gamma)$. The uniqueness of the asymptotic center yields
\begin{equation}\label{eqn:positivesides0}
\limsup_{k \tendsto \infty} d(x^{0},y^{k}) > \limsup_{k \tendsto \infty} d(y,y^{k}) = \sigma > 0.
\end{equation}
Again, by passing to a subsequence, we may assume that $(y^{k})$ has the following property for all $k \in \N$:
\begin{equation}\label{eqn:positivesides}
d(x^{0},y^{k}) > 0, \quad d(y,y^{k}) > 0, \quad \text{and $d(y,T_{s_{k}}y^{k}) > 0$}.
\end{equation}
For each $k \in \N$, let $\Delta(\bar{x^{0}},\bar{y},\bar{T_{s_{k}}y^{k}})$ be the $\kappa$-comparison triangle of $\Delta(x^{0},y,T_{s_{k}}y^{k})$ that share the common side $\llbracket \bar{x^{0}},\bar{y} \rrbracket$. In view of \eqref{eqn:positivesides0} and \eqref{eqn:positivesides}, the $\kappa$-angles $\angle^{(\kappa)}_{\bar{y_{k}}}(\bar{x^{0}},\bar{y})$, $\angle^{(\kappa)}_{\bar{y_{k}}}(\bar{x^{0}},\bar{T_{s_{k}}y^{k}})$, and $\angle^{(\kappa)}_{\bar{y_{k}}}(\bar{y},\bar{T_{s_{k}}y^{k}})$ exist, where $\bar{y^{k}}$ is the corresponding comparison point for $y^{k}$. We claim that $\angle^{(\kappa)}_{\bar{y_{k}}}(\bar{x^{0}},\bar{y}) \geq \pi/2$. Let us assume to the contrary that $\angle^{(\kappa)}_{\bar{y_{k}}}(\bar{x^{0}},\bar{y}) < \pi/2$. Since $\angle^{(\kappa)}_{\bar{y_{k}}}(\bar{x^{0}},\bar{T_{s_{k}}y^{k}}) = \pi$, this also implies that $\angle^{(\kappa)}_{\bar{y_{k}}}(\bar{y},\bar{T_{s_{k}}y^{k}}) \geq \pi/2$. On one hand, we have
\begin{align*}
\cos \sqrt{\kappa} d_{\kappa}(\bar{y},\bar{T_{s_{k}}y^{k}}) &= \cos \sqrt{\kappa} d_{\kappa}(\bar{y^{k}},\bar{T_{s_{k}}y^{k}})\cos \sqrt{\kappa} d_{\kappa}(\bar{y},\bar{y^{k}}) \\
&\qquad + \sin \sqrt{\kappa} d_{\kappa}(\bar{y^{k}},\bar{T_{s_{k}}y^{k}})\sin \sqrt{\kappa} d_{\kappa}(\bar{y},\bar{y^{k}})\cos \angle^{(\kappa)}_{\bar{y_{k}}}(\bar{x^{0}},\bar{y})\\
&< \cos \sqrt{\kappa} d_{\kappa}(\bar{y^{k}},\bar{T_{s_{k}}y^{k}}),
\end{align*}
which means $d_{\kappa}(\bar{y^{k}},\bar{T_{s_{k}}y^{k}}) < d_{\kappa}(\bar{y},\bar{T_{s_{k}}y^{k}})$. On the other hand, the fact that $y^{k} \preceq y$ gives
\begin{align*}
d_{\kappa}(\bar{y},\bar{T_{s_{k}}y^{k}}) &= d(y,T_{s_{k}}y^{k}) = d(T_{s_{k}}y,T_{s_{k}}y^{k}) \\
&\leq d(y,y^{k}) \leq d_{\kappa}(\bar{y},\bar{y^{k}}),
\end{align*}
which contradicts the earlier inequality. Therefore, it must be the case that $\angle^{(\kappa)}_{\bar{y_{k}}}(\bar{x^{0}},\bar{y}) \geq \pi/2$.

Again, by the $\kappa$-spherical law of cosines (Proposition \ref{prop:COSINE}), we have
\begin{align}\label{eqn:aux1}
\cos \sqrt{\kappa} d_{\kappa}(\bar{x^{0}},\bar{y}) &= \cos \sqrt{\kappa} d_{\kappa}(\bar{x^{0}},\bar{y^{k}})\cos \sqrt{\kappa} d_{\kappa}(\bar{y^{k}},\bar{y}) \nonumber\\
&\qquad + \sin \sqrt{\kappa} d_{\kappa}(\bar{x^{0}},\bar{y^{k}})\sin \sqrt{\kappa} d_{\kappa}(\bar{y^{k}},\bar{y}) \cos\angle^{(\kappa)}_{\bar{y_{k}}}(\bar{x^{0}},\bar{y}) \nonumber\\
&\leq \cos \sqrt{\kappa} d_{\kappa}(\bar{x^{0}},\bar{y^{k}})\cos \sqrt{\kappa} d_{\kappa}(\bar{y^{k}},\bar{y}).
\end{align}
Since $0 < d(x^{0},y^{k}) \leq d_{\kappa}(\bar{x^{0}},\bar{y^{k}})$, \eqref{eqn:aux1} further yields
\begin{equation}\label{eqn:star2}
d_{\kappa}(\bar{y^{k}},\bar{y}) < d_{\kappa}(\bar{x^{0}},\bar{y}).
\end{equation}

By the diameter assumption on $C$, the point
\[
u^{k} := \Proj_{\llbracket x^{0},y \rrbracket} y^{k}
\]
is well-defined for each $k \in \N$. Thus, $(u^{k})$ is a sequence in $\llbracket x^{0},y \rrbracket$. Since every geodesic interval is isometry to a compact interval in $\R$, we pass again to a subsequece and assume that $(u^{k})$ is strongly convergent to a point $u \in \llbracket x^{0},y \rrbracket$. Using the definitions of an asymptotic center and a projection, we obtain
\begin{align*}
\sigma &= \limsup_{k \tendsto \infty} d(y,y^{k}) \leq \limsup_{k \tendsto \infty} d(u,y^{k}) \\
&\leq \limsup_{k \tendsto \infty} d(u,u^{k}) + \limsup_{k \tendsto \infty} d(u^{k},y^{k}) \\
&= \limsup_{k \tendsto \infty} d(u^{k},y^{k}) \\
&\leq \limsup_{k \tendsto \infty} d(y,y^{k}) = \sigma.
\end{align*}
This shows $u = y$. Passing again to a subsequence, we may assume that $d(u^{k},y^{k}) > \sigma/2$ for all $k \in \N$. For $k \in \N$, let $\bar{u^{k}}$ and $\bar{u}$ b comparison points for $u^{k}$ and $u$, respectively, in the comparison triangle $\Delta(\bar{x^{0}},\bar{y},\bar{y^{k}})$ of $\Delta(x^{0},y,y^{k})$. Note that $u^{k} \neq x^{0}$ for all $k \in \N$. Otherwise, the Proposition \ref{prop:proj} gives
\[
\angle^{(\kappa)}_{\bar{x^{0}}} (\bar{y},\bar{y^{k}}) = \angle^{(\kappa)}_{\bar{u^{k}}} (\bar{y},\bar{y^{k}}) \geq \pi/2.
\]
Note that the angles above are defined in view of facts we derived earlier. By the $\kappa$-spherical law of cosines (Proposition \ref{prop:COSINE}), we subsequently get
\begin{align*}
\cos \sqrt{\kappa} d_{\kappa}(\bar{y},\bar{y^{k}}) &= \cos \sqrt{\kappa} d_{\kappa}(\bar{x^{0}},\bar{y}) \cos \sqrt{\kappa} d_{\kappa}(\bar{x^{0}},\bar{y^{k}}) \\
&\qquad + \sin \sqrt{\kappa} d_{\kappa}(\bar{x^{0}},\bar{y}) \sin \sqrt{\kappa} d_{\kappa}(\bar{x^{0}},\bar{y^{k}}) \cos \angle^{(\kappa)}_{\bar{x^{0}}} (\bar{y},\bar{y^{k}}) \\
&\leq \cos \sqrt{\kappa} d_{\kappa}(\bar{x^{0}},\bar{y}).
\end{align*}
This means $d_{\kappa}(\bar{x^{0}},\bar{y}) \leq d_{\kappa}(\bar{y},\bar{y^{k}})$, which contradicts with \eqref{eqn:star2}. Thus $u^{k} \neq x^{0}$ for all $k \in \N$. This shows that the angle $\gamma_{k} := \angle^{(\kappa)}_{\bar{u^{k}}} (\bar{x^{0}},\bar{y^{k}})$ is well-defined and the Proposition \ref{prop:proj} implies that $\gamma_{k} \geq \pi/2$ for all $k \in \N$. Apart from this, we also define for each $k \in \N$ the following quantities:
\[
a_{k} := d_{\kappa} (\bar{x^{0}},\bar{u^{k}}), \quad b_{k} := d_{\kappa} (\bar{u^{k}},\bar{y^{k}}), \quad \text{and $c_{k} := d_{\kappa} (\bar{x^{0}},\bar{y^{k}})$}.
\]
We may see now that
\[
\sigma/2 < b_{k} \leq c_{k} < d_{\kappa}(\bar{x^{0}},\bar{y})
\]
at each $k \in \N$. By the $\kappa$-spherical law of cosines (Proposition \ref{prop:COSINE}), we obtain
\begin{align*}
\cos \sqrt{\kappa} c_{k} &= \cos \sqrt{\kappa} a_{k} \cos \sqrt{\kappa} b_{k} + \sin \sqrt{\kappa} a_{k} \sin \sqrt{\kappa} b_{k} \cos \gamma_{k} \\
&\leq \cos \sqrt{\kappa} a_{k} \cos \sqrt{\kappa} b_{k}.
\end{align*}
The two inequalities above implies
\[
\cos \sqrt{\kappa} a_{k} \geq \frac{\cos \sqrt{\kappa} c_{k}}{\cos \sqrt{\kappa} b_{k}} > \frac{\cos \sqrt{\kappa} d_{\kappa}(\bar{x^{0}},\bar{y})}{\cos \sqrt{\kappa} (\sigma/2)} > \cos \sqrt{\kappa} d_{\kappa}(\bar{x^{0}},\bar{y}).
\]
For convenience, we put
\[
\delta := \frac{1}{\sqrt{\kappa}}\arccos \left(\frac{\cos \sqrt{\kappa} d_{\kappa}(\bar{x^{0}},\bar{y})}{\cos \sqrt{\kappa} (\sigma/2)}\right).
\] Note that $\delta$ is independent of $k \in \N$. Hence, we get $a_{k} < \delta < d_{\kappa}(\bar{x^{0}},\bar{y})$ and then
\begin{align*}
d(y,u^{k}) &= d_{\kappa}(\bar{y},\bar{u^{k}}) = d_{\kappa}(\bar{x^{0}},\bar{y}) - d_{\kappa}(\bar{x^{0}},\bar{u^{k}}) \\
&> d_{\kappa}(\bar{x^{0}},\bar{y}) - \delta > 0.
\end{align*}
This shows that $(d(y,u^{k}))$ is bounded away from $0$, which together implies that $u \neq y$. This is a contradiction. Therefore, the sequence $(y^{k})$ is convergent to $y$. Since all subsequence of $(x^{k})$ contains a subsequent convergent to $y$, we conclude that $(x^{k})$ converges to $y \in \Fix(\Gamma)$. Since $(y^{k})$ is $\preceq$-nondecreasing, we have $x^{0} \preceq y$.

Next, we show the second conclusion. SUppose that $q \in \Fix(\Gamma)$ satisfies $v^{k} \preceq q$ at each $k \in \N$, for some subsequence $(v^{k})$ of $(x^{k})$. Let $(\beta_{k})$ and $(s_{k})$ be the subsequences of $(\lambda_{k})$ and $(t_{k})$, respectively, in which $v^{k} = (1-\beta_{k})T_{s_{k}}v^{k} \oplus \beta_{k} x^{0}$ for $k \in \N$. We may also assume that $v^{k} \neq x^{0}$ at all $k \in \N$. For each $k \in \N$, let $\Delta(\bar{q},\bar{x^{0}},\bar{T_{s_{k}}v^{k}})$ be the comparison triangle of $\Delta(q,x^{0},T_{s_{k}}v^{k})$ that share the common side $\llbracket \bar{q},\bar{x^{0}} \rrbracket$. Observe that we have $d(T_{s_{k}}v^{k},q) \leq d(v^{k},q)$. If $\angle^{(\kappa)}_{\bar{v^{k}}} (\bar{q},\bar{T_{s_{k}}v^{k}}) > \pi/2$, we further have
\begin{align*}
\cos \sqrt{\kappa} d_{\kappa}(\bar{q},\bar{T_{s_{k}}v^{k}}) &= \cos \sqrt{\kappa} d_{\kappa}(\bar{q},\bar{v^{k}}) \cos \sqrt{\kappa} d_{\kappa}(\bar{v^{k}},\bar{T_{s_{k}}v^{k}}) \\
&\qquad + \sin \sqrt{\kappa} d_{\kappa}(\bar{q},\bar{v^{k}}) \sin \sqrt{\kappa} d_{\kappa}(\bar{v^{k}},\bar{T_{s_{k}}v^{k}}) \cos \angle^{(\kappa)}_{\bar{v^{k}}} (\bar{q},\bar{T_{s_{k}}v^{k}})\\
&< \cos \sqrt{\kappa} d_{\kappa}(\bar{q},\bar{v^{k}}) \leq \cos \sqrt{\kappa} d_{\kappa}(\bar{q},\bar{T_{s_{k}}v^{k}}),
\end{align*}
which is absurd. Hence, it must be the case that $\angle^{(\kappa)}_{\bar{v^{k}}} (\bar{q},\bar{T_{s_{k}}v^{k}}) \leq \pi/2$. If follows that $\angle^{(\kappa)}_{\bar{v^{k}}} (\bar{x^{0}},\bar{q}) > \pi/2$. Again, from the $\kappa$-spherical law of cosines (Proposition \ref{prop:COSINE}), we have
\begin{align*}
\cos \sqrt{\kappa} d_{\kappa}(\bar{x^{0}},\bar{q}) &= \cos \sqrt{\kappa} d_{\kappa}(\bar{q},\bar{v^{k}}) \cos \sqrt{\kappa} d_{\kappa}(\bar{x^{0}},\bar{v^{k}}) \\
&\qquad + \sin \sqrt{\kappa} d_{\kappa}(\bar{q},\bar{v^{k}}) \sin \sqrt{\kappa} d_{\kappa}(\bar{x^{0}},\bar{v^{k}}) \cos \angle^{(\kappa)}_{\bar{v^{k}}} (\bar{x^{0}},\bar{q})\\
&\leq \cos \sqrt{\kappa} d_{\kappa}(\bar{x^{0}},\bar{v^{k}}).
\end{align*}
Subsequently, we may ses that
\[
d(x^{0},v^{k}) = d_{\kappa}(\bar{x^{0}},\bar{v^{k}}) \leq d_{\kappa}(\bar{x^{0}},\bar{q}) = d(x^{0},q).
\]
The final conclusion follows by letting $k \tendsto \infty$.
\end{proof}

\section{Conclusion and Remarks}

As a quick summary, we have established an existence theorem for the class of $\preceq$-nonexpansive semigroups. Then, we proposed two approximation schemes, the Krasnosel'ski\u\i{}'s and the Browder's. The first is explicit but works only with discrete semigroup while the second is implicit but works in any semigroups. However, there are still limitations in terms of generality. In the non-ordered case (over both linear and nonlinear spaces), the choices of parameter sequences $(\lambda_{k})$ and $(t_{k})$ are more freely available. Based on these inspections, we shall pose here the following open questions.
\begin{ques}
How to generalize parameter conditions on $(t_{k})$ of the Krasnosel'ski\u\i{} approximation to any semigroups not nocessarily discrete?
\end{ques}
\begin{ques}
How to generalize the parameter conditions on $(\lambda_{k})$ and $(t_{k})$ in the Browder approximation?
\end{ques}

\section*{Acknowledgements}

The author is thankful to Professor Mohamed Amine Khamsi for his helpful advice over the topic of this paper during his visit at KMUTT in November 2018.

\small
\renewcommand\bibname{References}

\end{document}